\documentclass[a4paper,10pt]{article}

\usepackage{amsmath, amssymb}
\usepackage{amsthm}
\usepackage{tipa}

\newtheorem{thm}{Theorem}[subsection]
\newtheorem{dfn}[thm]{Definiton}
\newtheorem{lem}[thm]{Lemma}
\newtheorem{rem}[thm]{Remark}
\newtheorem{cor}[thm]{Corollary}
\newtheorem{prop}[thm]{Proposition}
\title{Counter Examples in Non-Archimedean Locally Convex Spaces}
\date{}
\begin{document}
	\maketitle
\begin{center}
 M.E. Egwe\\
Department of Mathematics,\\
Universty of Ibadan, Ibadan.\\
$murphy.egwe@mail.ui.edu.ng$\\
\ \\	
 J. A Braimah\\
  McPherson University, Seriki, Sotayo,\\
  Ajebo, Ogun State, NIGERIA.\\
  $braimahjoachim@gmail.com$
\ \\
	\end{center}	
	
	\begin{abstract}
		In this paper, we shall consider some counter examples in non-archimedean locally convex spaces with special closed subspaces and Schauder basis in non-archimedean Fr\'{e}chet spaces as well as closed subspaces \emph{without} Schauder basis in non-archimedean Fr\'{e}chet spaces.
	\end{abstract}
\section{Introduction, Preliminary Concepts and Definitions}
The analogy with Schauder basis is seen to be more appropriate than using the notion of orthogonality in non-archimedean spaces as it has many of the properties that its namesake in classical Hilbert spaces has. We have that every orthogonal basis in a metrizable locally convex spaces $X$ is a Schauder basis in $X$ and every Schauder basis in Fr\'{e}chet space $X$ is an orthogonal basis in $X$. Every bases in a Fr\'{e}chet space is Schauder. Clearly, every locally convex space with a (Schauder) base is (strictly) of countable type. That is, there is a countable set whose $K-$linear hull is dense in $X$. Conversely, any infinite-dimensional Banach space of countable type has Schauder basis, since it is known to be linearly homeomorphic (i.e isomorphic ) to the Banach space $C_0$ of all sequences in $K$ converging to zero. For Fr\'{e}chet spaces, it is not true \cite{SlW3}. Hence, there exists an infinite-dimensional Fr\'{e}chet spaces of countable type without a Schauder basis \cite{SlW5}. With construction of examples of nuclear Fr\'{e}chet spaces without a Schauder basis in \cite{SlW2} we solved the problem whether any Fr\'{e}chet space of countable type has a Schauder basis.

Nevertheless, any infinite-dimensional Fr\'{e}chet space $F$ of finite type is isomorphic to the Fr\'{e}chet space $\mathbb{K}^N$ of all sequences in $K$ (with the topology of pointwise convergence). So every closed subspace of $F$ has Schauder basis.
Moreover, any infinite-dimensional Fr\'{e}chet space contains an infinite-dimensional closed subspace with a Schauder basis \cite{SlW1}.\\
It is also known that any closed subspace of $C_0 \times \mathbb{K}^N$ has a Schauder basis \cite{SlW3}. On the other hand infinite-dimensional Fr\'{e}chet spaces $C_0,\mathbb{K}^N, C_0 \times \mathbb{K}^N$ contains a closed subspace without a Schauder basis \cite{SlW3}.

In \cite{SlW3}, A Fr\'{e}chet space $X$ which is not of finite type such that none of the subspace is isomorphic to $C_0$ is seen, from which we establish the ideas that $X$ contains infinitely many pairwise non-isomorphic closed subspaces ( without a Schauder basis ) with a strong finite-dimensional Schauder decomposition \cite{SlW3}.

\begin{dfn}
	A sequence $(x_n)$ in $X$ is a Schauder basis of $X$ if each $x \in X$ can be written uniquely as $$x = \sum_{n=1}^{\infty} a_nx_n$$
	with $(a_n) \subset \mathbb{K}$ and the coefficient functionals $x_{n}^{*}: X \rightarrow \mathbb{K} , x \rightarrow a_n, n\in \mathbb{N}$ are continuous.
	
	A sequence $(x_n)$ in $X$ is a Schauder basic sequence in $X$ if it is a Schauder basis of the closed linear span in $X$. Every base in a Fr\'{e}chet space is Schauder. Clearly every locally convex space with a base (Schauder ) is strictly of countable type.
\end{dfn}

\begin{dfn}
	Let $X$ be a locally convex space and $L(X)$ be a space of continuous maps from $X$ to itself. A sequence $(A_n) \subset L(X)$ is a Schauder partition of $X$ if $x = \sum_{n=1}^{\infty} A_n x$. A Schauder partition $(A_n)$ of $X$ is $\gamma-$finite-dimensional ($\gamma \in \mathbb{N}$) if $\sup_n \dim A_n(X) \leq \gamma$, and strong-finite-dimensional if $\sup_n \dim A_n(X) < \infty $, and finite-dimensional if $\dim A_n(X) < \infty $ for all $n \in \mathbb{N}$
	
\end{dfn}

A Schauder partition $(P_n)$ of a locally convex space $X$ is a Schauder decomposition of $X$ if $P_nP_m = \delta_{nm}P_n$ for all $n,m \in \mathbb{N}, P_nP_m = 0 $ if $n \neq m$.\\
Clearly, any locally convex space $X$ with a Schauder basis has strong finite-dimensional Schauder decomposition.

\begin{dfn}[\cite{VA1}]
	Let $p$ be a seminorm on a linear space $X$ and $t \in (0,1]$. An element $x \in X$ is $t-$orthogonal to a subspace $M$ of $X$ with respect to $p$ if
	$$p(ax + y) \geq t\max\{p(ax),p(y) \}\quad \forall a\in \mathbb{K}, y \in M$$
	
	A sequence $(x_n) \subset X$ is $t-$orthogonal with respect to $p$ if
	$$p\left( \sum_{i=1}^{n} a_i x_i \right) \geq t \max p(a_i x_i) \quad x \in \mathbb{N} \text{ and } a_1, \ldots , a_n \in \mathbb{K}$$
	
	Also,let $t_k \in (0,1]$. A sequence $(x_n)$ in a metrizable locally convex space $X$is $t_k-$orthogonal with respect to $p_k$ for every $k \in \mathbb{N}$. (If $t_k= 1$ for $k \in \mathbb{N}$, then we shall write $1-$orthogonal instead of $(1)-$orthogonal)
	
	A sequence $(x_n)$ in a metrizable locally convex space $X$ is orthogonal if it is $1-$orthogonal with respect to some base $(p_k)$ in $P(X)$.
\end{dfn}

\begin{dfn}
	$X$ is called a normable if its topology $\tau$ is generated by a norm. It is called metrizable if there is a metric on $X$ inducing the topology $\tau$.
\end{dfn}
\section{Closed Subspaces with Schauder bases in non-archimedean Fr\'{e}chet Spaces}
 It is well known that any infinite-dimensional Banach space of countable type has Schauder basis. For Fr\'{e}chet spaces. It is not in general, true \cite{SlW3}, hence there exist infinite-dimensional Fr\'{e}chet spaces of countable type without Schauder basis \cite{SlW5}. Also, any infinite-dimensional Fr\'{e}chet space $F$ of finite type is isomorphic to the Fr\'{e}chet space $\mathbb{K}^N$ for all sequences in $\mathbb{K}$ (with the topology of pointwise convergence) so every closed subspace of $F$ has Schauder basis.

In what follows, we shall be considering first, normable closed subspaces. We shall show that a Fr\'{e}chet space is normable  if and only if each of its closed subspaces with Schauder basis is normable \cite{SlW4}, then we show that a Fr\'{e}chet space with a Schauder basis $(x_n)$ has a subsequence $(x_{n_k})$ whose closed linear span is isomorphic to $C_0$ \cite{SlW4}.
On normable closed subspace, we have the following results.

\begin{lem}[\cite{SlW4}] \label{lem1}
	Let $n \in \mathbb{N}$ and let $p_1, \ldots, p_n$ be continuous seminorms on a metrizable locally convex space $X$ of countable type. Let $M$ be a finite-dimensional subspace of $X$. Then for every $t \in (0,1]$ there exists a closed subspace $L$ of $X$ with $\dim (X/L) < \infty $ such that any $x \in L$ is $t-$orthogonal to $M$ with respect to $p_i$ for all $1 \leq i \leq n$
\end{lem}

\begin{proof}
	Let $1 \leq i \leq n$ and $Y_i = X/\ker p_i$, let $\pi_i : X \rightarrow Y_i$ be the quotient mapping and denote by $(a_i \tilde{p}_i)$ of countable type, then there exists a linear continuous projection $Q_i$ of $G_i$ onto $\pi _i(M)$ of norm less than or equal to $t^{-1}$\cite{VA1}
	
	Let $H_i = Y_i \cup \ker Q_i$ and $X_i = \pi _1 ^{-1} (H_i)$\\
	Any $x \in X_i$ is $t-$orthogonal to $M$ with respect to $p_i$, indeed let $\alpha \in \mathbb{K}, m \in M, z  = \pi _i(m)$ and $y = \pi _i(x)$. Since $z = Q_i (\alpha y + z)$, then $\tilde{p}_i(z) \leq t^{-1} \tilde{p}_i(\alpha y + z)$. Hence $$ \tilde{p}_i(\alpha y + z) \geq t \max \{ \tilde{p}_i(\alpha y), \tilde{p}_i(z) \} \cite{VA1}$$
	Thus
	$$p_i(\alpha x + m) \geq t \max \{ p_i(\alpha x), p_i(m) \}$$
	Let $L = \bigcap _{i=1}^{n} X_i$. Any $x \in L$ is $t-$orthogonal to $M$ with respect to $p_i$ for all $1 \le i \leq n$. Clearly, $L$ is a closed subspace of $X$ and
	$$\dim (X/L) \leq \sum_{i=1}^{n} \dim (X/X_i) = \sum_{i=1}^{n} \dim (Y_i/H_i)$$
	$$ \sum_{i=1}^{n} \dim (G_i/\ker Q) < \infty $$
\end{proof}

\begin{lem} \label{lem2}
	Let $X$ be a metrizable locally convex space with a base $(p_k)$ in $P(X)$. Assume that $(s_n) \subset (0,1)$ with $s:= \prod _{n=1}^{\infty} s_n >0$. Then any sequence $(y_n) \subset (X \ker p_i)$ such that $y_{n+1}$ is $s_{n+1}-$orthogonal to $\lim\{ y_1, \ldots, y_n \}$ with respect to $p_i$ for all $1 \le i \le n$ and $n \in \mathbb{N}$, is orthogonal in $X$.
\end{lem}

\begin{proof}
	It is enough to show that the sequence $(y_n)$ is $(t_m)-$orthogonal with respect to $(p_m)$ for some $(t_m) \subset (0,1]$.\\
	Let $m \in \mathbb{N}$ and $\alpha _1, \ldots, \alpha _k \in \mathbb{K}$. Then
	$$p_1 \left( \sum_{i=1}^{m} \alpha _i y_i \right) \geq d_m \max_{1 \le i \le m} p_m(\alpha_iy_i)  $$
	Let $k > m$ and $\alpha_1,\ldots , \alpha_k \in \mathbb{K}$. Then
	$$p_m \left( \sum_{i=1}^{k} \alpha_iy_i \right) \ge \left( \prod_{i=m+1}^{k} s_i \right) d_m \max_{1 \le i \le k} (\alpha_iy_i) \ge s d_m \max_{1 \le i \le k}p_m (\alpha_iy_i)$$
	
	Thus the sequence $(y_n)$ is $(sd_m)-$orthogonal with respect to $(p_m)$. Hence, we can deduce that.
\end{proof}

\begin{cor}[\cite{SlW6}]
	Any infinite-dimensional Fr\'{e}chet space has a Schauder basic sequence, we have our first theorem in what follows.
\end{cor}

\begin{thm}
	A Fr\'{e}chet Space is normable if and only if each of the closed subspaces with s Schauder basis is normable.
\end{thm}

\begin{proof}
	It is enough to show that any non-normable Fr\'{e}chet space $X$ contains a non-normable closed subspace $\mathcal{G}$ with a Schauder basis.\\
	Considering two cases.\\
	\textbf{\underline{Case 1}} $X$ has a continuous norm\\
	$X$ contains a non-normable closed subspace $Y$ of countable type. Let $(\mathbb{N}_k)$ be a sequence of pairwise disjoint infinite sets with $\bigcup_{k=1}^{\infty} \mathbb{N}_k = \mathbb{N}$ and let $(s_n) \subset (0,1)$ with $\prod_{n=1}^{\infty} s_n > 0$; there exists a base of norms $(p_k)$ in $P(Y)$ such that for any $k \in \mathbb{N}$ the norms $p_k$ and $p_{k+1}$ are non-equivalent on any subspace of finite co-dimension in $Y$. Then using lemma \ref{lem1} we can construct inductively a sequence $(x_n) \subset Y$ such that $x_{n+1}$ is $s_{n+1}-$orthogonal to $\lim \{ x_1, \ldots , x_n \}$ with respect to $p_i$ for $1 \le i \le n, n \in \mathbb{N}$, and $p_k(x_n) < n^{-1} p_{k+1} (x_n)$ for all $n \in \mathbb{N}_k$
	
	By lemma \ref{lem2}, $(x_n)$ is orthogonal in $Y$. Clearly,
	$$ \inf _{n \in \mathbb{N}} [p_k(x_n)/p_{k+1}(x_n)] = 0 \quad \text{ for any } k \in \mathbb{N}$$
	Hence, for every $k \in \mathbb{N}$ the norms $p_k$ and $p_{k+1}$ are non-equivalent on the closed linear span $\mathcal{G}$ of $(x_n)$.\\

Thus, $\mathcal{G}$ is a non-normable closed subspace with a Schauder basis in $X$.\\
\textbf{\underline{Case 2.}} $X$ has no continuous norm.\\
Then $X$ contains a closed subspace $\mathcal{G}$ is isomorphic to $\mathbb{K}^N$ so, it has non-normable closed subspace with a Schauder basis. We see that a Fr\'{e}chet space with a Schauder basis possesses an infinite dimensional normable closed subspace with a Schauder basis from the following proposition.
\end{proof}

\begin{prop}[\cite{SlW4}]
	A Fr\'{e}chet space $X$ with a Schauder basis $(x_n)$ contains a subspace isomorphic to $C_0$ if and only if $(x_n)$ has a subsequence $(x_{k_n})$ whose closed linear span is isomorphic to $C_0$
\end{prop}

\begin{rem}
	Let $(x_n)$ be a Schauder basis in a Fr\'{e}chet space $X$. Assume that $(x_n)$ is $1-$orthogonal with respect to a base $(p_k)$ in $P(X)$. Then $(x_n)$ has a subsequence $(x_{k_n})$ whose closed linear span is isomorphic to $C_0$ if and only if there exist an infinite subset $M$ of $\mathbb{N}$, a sequence $(d_k) \subset (0,1)$ and $K_0 \in \mathbb{N}$ such that $p_k(x_n) \ge d_{k+1}p_{k+1} (x_n) > 0$ for all $K \ge K_0$ and $n \in M$. Clearly, any Fr\'{e}chet space which contains a closed subspace isomorphic to $C_0$ is non-nuclear, but the converse is not true.
\end{rem}

\section{Closed Subspaces without Schauder Base in non-archimedean Fr\'{e}chet Space}

Next, we look at existence of closed subspaces without Schauder basis in Fr\'{e}chet spaces.
Let $X$ be a Fr\'{e}chet Space which is not of finite type such that none of its subspaces is isomorphic to $C_0$. We shall establish from the ideas of \cite{SlW2} that $X$ contains infinitely many pairwise non-isomorphic closed subspaces (without a Schauder basis ) with a strong finite-dimensional Schauder decomposition \cite{SlW3}. Thus, we clearly obtain the following:
\begin{prop}[\cite{SlW3}]
	Any non-normable Fr\'{e}chet space $Y$ with a continuous norm and with a Schauder basis contains an infinite dimensional closed subspaces without a strongly finite dimensional Schauder  decomposition. Hence, in other words, there exist nuclear Fr\'{e}chet spaces with a strong finite dimensional Schauder decomposition but without a Schauder basis.
\end{prop}
\begin{rem}
It is a fact that an infinite dimensional Fr\'{e}chet space $X$ of countable type that is not isomorphic to any of the following spaces $C_0, C_0 \times \mathbb{K}^N, \mathbb{K}^N$. Then $X$ contains a non-normable closed subspace with a continuous norm, and it has been proved that $X$ being a non-normable Fr\'{e}chet space with a continuous norm and with a finite dimensional Schauder decomposition $(p_n)$, then $X$ contains a non-normable closed subspace $Y$ with a Schauder basis \cite{SlW3}. With these facts we obtained the following
\end{rem}
\begin{thm}
	Let $X$ be an infinite-dimensional Fr\'{e}chet space, which is no isomorphic to any of the following spaces: $C_0,C_0 \times \mathbb{K}^N, \mathbb{K}^N$. Then $X$ contains an infinite-dimensional closed subspace without a strongly finite-dimensional Schauder decomposition. Finally, we obtain the following
\end{thm}

\begin{prop}
	Let $X$ be an infinite-dimensional metrizable locally convex space whose completion $X$ is not isomorphic to any of the following spaces $C_0,C_0 \times \mathbb{K}^N, \mathbb{K}^N$. Then $X$ contains a closed subspace without an orthogonal Schauder basis
\end{prop}

\begin{proof}
	It is enough to consider the case when $X$ has an orthogonal (Schauder basis ) $(y_m)$. Then $(y_n)$ is an orthogonal Schauder. Let $(q_k)$ be a non-decreasing base in $P(X)$ with $q_1 = 0$ such that $(y_n)$ is $(1)-$orthogonal with respect to $(q_k)$
	
	Put $D_k = \{ n \in\mathbb{N} : y_n \in (\ker q_{k-1})\backslash \ker q_k \}, k \in \mathbb{N}$. Denote by $Y_k$ the closed linear span of $\{ y_n : n \in D_k \}$ in $Y$. As in the proof of lemma 5\cite{SlW3}. We obtain that $Y$ is isomorphic to $\prod_{n=1}^{\infty} Y_n$. If all the Fr\'{e}chet spaces $Y_k, k \in \mathbb{N}$ are normable, then $Y$ is isomorphic to one of the following spaces: $C_0,C_0 \times \mathbb{K}^N, \mathbb{K}^N$, contrary to our assumption. Thus for some $m \in \mathbb{N}$, the space $Y_m$ is non-normable.
	
	Clearly, $\{ y_n, n \in D_m \}$ is an orthogonal Schauder basis of $Y_m$ and $q_m \mid Y_m$ is a continuous norm on $Y_m$. Hence by the proofs of (Lemma 1 and theorem 2  of \cite{SlW3}) there exists a linear subspace $V$ of $\lim \{ y_n : x \in DF_m \}$ such that the closure $V_1$ of $V$ in $Y$ has no Schauder basis. Then the closure $V_2$ of $V$ in $X$ has no orthogonal Schauder basis.
\end{proof}

Finally we will conclude this section by stating the following

\begin{rem}.
	\begin{enumerate}
		\item[-1] There exist nuclear Fr\'{e}chet space with a strong finite dimensional Schauder decomposition but without a Schauder basis.
		\item[-2] There exist nuclear Fr\'{e}chet space with a finite-dimensional Schauder decomposition but without a strong finite-dimensional Schauder decomposition.
	\end{enumerate}
\end{rem}

\end{document}